\documentclass[12pt,twoside]{amsart}
\usepackage{geometry}
\usepackage{amsmath,amssymb,amsthm,amscd}
\usepackage{mathtools}
\usepackage{mathrsfs}
\usepackage[shortlabels]{enumitem}
\usepackage{chngcntr}

\usepackage{graphicx}
\usepackage{tabularray}
\usepackage{float}
\usepackage{hhline}
\usepackage{tikz}
\usepackage{tikz-cd}

\usepackage{xcolor}

\usepackage[hidelinks]{hyperref}

\newif\ifshowchanges
\showchangestrue
\ifdefined\CleanVersion\showchangesfalse\fi

\usepackage{etoolbox}
\newcommand{\revisionbibcolor}[1]{%
  \normalcolor
  \ifshowchanges
  \fi}
\makeatletter
\AtBeginDocument{%
  \let\revision@original@lbibitem\@lbibitem
  \def\@lbibitem[#1]#2{%
    \revisionbibcolor{#2}%
    \revision@original@lbibitem[#1]{#2}}}
\makeatother
\AtEndEnvironment{thebibliography}{\normalcolor}

\setlist[enumerate]{leftmargin=56pt,labelsep=8pt,itemsep=4pt,label=\upshape{(\theequation.\arabic*)}}

\makeatletter
\renewcommand{\theequation}{%
\thesection.\arabic{equation}}
\@addtoreset{equation}{section}
\makeatother

\newtheorem{thm}{Theorem}[section]

\newtheorem{cor}[thm]{Corollary}

\theoremstyle{definition}

\newtheorem{rem}[thm]{Remark}

\newtheorem{fact}[thm]{Fact}
\newtheorem{step}{Step}
\newtheorem*{ack}{Acknowledgments}

\title{Non-projective complete log canonical surfaces}
\author{Osamu Fujino, Nao Moriyama, Hiroshi Sato} 
\address{Department of Mathematics, Graduate School of Science, Kyoto University, Kyoto 606-8502, Japan}
\email{fujino@math.kyoto-u.ac.jp}
\address{Department of Mathematics, Graduate School of Science, Kyoto University, Kyoto 606-8502, Japan}
\email{moriyama.nao.22s@st.kyoto-u.ac.jp}
\address{Department of Applied Mathematics, Faculty of Sciences, Fukuoka University, 8-19-1, Nanakuma, Jonan-ku, Fukuoka 814-0180, Japan
}
\email{hirosato@fukuoka-u.ac.jp}
\date{2026/9/29, version 0.07}

\makeatletter
\renewcommand{\@biblabel}[1]{\makebox[3.5em][l]{[#1]}}
\makeatother

\begin{document}

\begin{abstract}
We construct non-projective complete log canonical algebraic surfaces whose canonical divisor is semi-ample over any algebraically closed field other than the algebraic closure of a finite field. We provide a unified framework to construct such surfaces for any given non-negative 
Kodaira dimension, namely, zero, one, or two. 
Furthermore, we give a relatively elementary proof, within the framework of the minimal model theory, of the essentially known fact that any complete log canonical surface with Kodaira dimension minus infinity is projective. This fact confirms that our construction covers all possible Kodaira dimensions. As a byproduct, the proof shows that such a surface of Kodaira dimension minus infinity has at most one non-rational singular point. 
\end{abstract}

\maketitle

\section{Introduction}
In \cite{nagata}, Nagata constructed the first example of a non-projective complete normal algebraic surface, which is a landmark result in the development of algebraic geometry. The surface he constructed has very bad singularities from the viewpoint of the minimal model theory. In this paper, we revisit the existence problem of non-projective complete normal algebraic surfaces from the perspective of the minimal model theory.

Let $k$ be an algebraically closed field and let $X$ be a complete normal algebraic surface defined over $k$. 
It is easy to see that $X$ is always projective if $X$ is $\mathbb{Q}$-factorial. 
Since $X$ is known to be $\mathbb{Q}$-factorial when it has only rational singularities, $X$ is automatically projective in this case as well. 
Consequently, if $X$ has only log terminal singularities, it is always projective. 
Furthermore, it is a classical result of Artin that $X$ is always $\mathbb{Q}$-factorial when $k=\overline{\mathbb{F}}_p$, which implies that $X$ is always projective in this setting. 
On the other hand, it is well known that $X$ is not necessarily $\mathbb{Q}$-factorial if we allow $X$ to have log canonical singularities. 
Therefore, it is natural to ask whether there exists a non-projective complete log canonical algebraic surface. 
Although this problem has already been discussed in \cite[Section 12]{fujino-nagoya}, here we would like to take a step further and establish the following theorem.

\begin{thm}\label{thm1.1}
Let $k$ be an algebraically closed field such that $k \ne \overline{\mathbb{F}}_p$. 
Then there exists a non-projective complete log canonical algebraic surface $X$ 
defined over $k$ such that $K_X$ is semi-ample. 
Furthermore, the Kodaira dimension $\kappa(X, K_X)$ can be chosen to be any value in $\{0, 1, 2\}$. 
\end{thm}

Here and throughout this paper, by the Kodaira dimension of $X$ we mean $\kappa(X, K_X)$, the Iitaka dimension of $K_X$ on $X$. Note that this need not coincide with the Kodaira dimension of a resolution of $X$. 

Theorem 1.1 is perfectly complemented by the following theorem. 

\begin{thm}\label{thm1.2}
Let $X$ be a complete log canonical algebraic surface over an algebraically closed field.
If $\kappa(X,K_X)=-\infty$, then $X$ is projective.
\end{thm}

Although Theorem \ref{thm1.2} follows easily from Schr\"oer's projectivity criterion \cite[Corollary 4.4]{Sch00}, we include a relatively elementary proof of Theorem~\ref{thm1.2} within the framework of the minimal model program. As a byproduct, the proof shows that $X$ has at most one non-rational singular point (Corollary~\ref{cor2.2}). 

To clarify the position of this work within the framework of the minimal model theory, let us briefly review the context. 
The first author and Tanaka \cite{fujino-surface, tanaka} established the minimal model theory for log surfaces, proving that the MMP and the abundance theorem hold for both projective $\mathbb{Q}$-factorial log surfaces and projective log canonical surfaces.
Since a complete $\mathbb{Q}$-factorial log surface is always projective, attention naturally turns to the existence of non-projective complete log canonical surfaces.
While such an example with $\kappa(X, K_X) = 0$ was previously constructed in \cite[Section 12]{fujino-nagoya}, it is natural to ask whether such examples exist for other Kodaira 
dimensions, particularly for general type where $\kappa(X, K_X) = 2$.
Theorems~\ref{thm1.1} and \ref{thm1.2} provide 
a clear picture of this situation.
Together with \cite{Mor26}, this work provides insight into the scope and limitations of the minimal model theory for log surfaces. 
Finally, we remark that our construction is based on the examples presented in \cite{Sch99}; the examples with $\kappa(X,K_X)=1$ and $2$ are obtained by taking cyclic covers. 
Note that while Schr\"oer \cite{Sch99} assumes the base field $k$ to be uncountable, our arguments work over any algebraically closed field $k \ne \overline{\mathbb{F}}_p$. We consider this removal of the uncountability assumption to be a minor but novel point. As emphasized earlier, we treat Theorems~\ref{thm1.1} and \ref{thm1.2} systematically within the framework of the minimal model theory. 

\begin{ack} The first author was partially supported by JSPS KAKENHI Grant Number JP23K20787. The second author was partially supported by JSPS KAKENHI Grant Number JP26KJ1528. The third author was partially supported by JSPS KAKENHI Grant Number JP24K06679. 
The authors would like to acknowledge the assistance of Rethlas and
ChatGPT in the exploratory phase of this work. Their suggestions led
the authors to consider examples that ultimately motivated the results
presented here. All proofs and mathematical verifications were carried
out independently by the authors. 
The authors would like to thank the anonymous reviewer of a previous version of this manuscript for their valuable comments and suggestions. 
\end{ack}

\section{Preliminaries} 
In this section, we collect some basic results. For the details on the definitions of singularities of pairs, see \cite[Section 2.4]{fujino-surface}. A boundary $\mathbb{Q}$-divisor is an effective $\mathbb{Q}$-divisor whose coefficients lie in $[0,1]$. Since we work within the framework of the minimal model program, all surfaces constructed below are algebraic varieties; in particular, our constructions yield schemes. Note that we do not use algebraic spaces in this paper.

The following fact is used to construct an example verifying Theorem~\ref{thm1.1} for $k\neq\overline{\mathbb{F}}_p$.
\begin{fact}[{cf.~\cite[Fact 2.3]{tanaka}}]\label{fact: non-torsion}
Let $k$ be an algebraically closed field of arbitrary characteristic and let $C$ be an elliptic curve over $k$. If $k\neq\overline{\mathbb{F}}_p$, then $\operatorname{Pic}^0(C)$ has a non-torsion element.
\end{fact}

In this paper, to construct contraction morphisms, we rely on two tools available for log canonical surfaces: the relative abundance theorem and the basepoint-free theorem.
\begin{thm}[{Relative abundance for log canonical log surfaces}]\label{thm: abundance}
Let $k$ be an algebraically closed field of arbitrary characteristic. Let $T$ be a normal surface over $k$ and $\Delta$ be a boundary $\mathbb{Q}$-divisor on $T$ such that $K_T+\Delta$ is $\mathbb{Q}$-Cartier.
Let $\pi : T \to Y$ be a projective surjective morphism onto a variety $Y$. 
Assume that $(T,\Delta)$ is log canonical. 
Further, assume that $K_T + \Delta$ is $\pi$-nef. 
Then $K_T + \Delta$ is $\pi$-semi-ample.
\end{thm}
\begin{proof}
The case where $k=\mathbb{C}$ is given by \cite[Theorem 7.2]{fujino-surface}, while the positive characteristic case is covered by \cite[Theorem 6.9]{tanaka}. For the remaining cases in characteristic zero, one can refer to \cite[Theorem 4.1]{Fuj21}.
\end{proof}
\begin{thm}[{Basepoint-freeness for non-klt log surfaces, see \cite[Theorem 3.1]{Mor26}}]\label{thm: basepoint-free}
Let $T$ be a projective normal surface over an algebraically closed field $k$ of arbitrary characteristic and let $\Delta$ be an effective $\mathbb{Q}$-divisor such that $K_T + \Delta$ is $\mathbb{Q}$-Cartier. 
Let $D$ be a nef Cartier divisor.
If $\mathrm{char}(k)>0$, we further assume that $D$ is not numerically trivial. 
Suppose that $aD - (K_T + \Delta)$ is nef and big for some $a \in \mathbb{Z}_{>0}$, and that the restriction $D|_{\mathrm{Nklt}(T, \Delta)}$ is semi-ample, where $\mathrm{Nklt}(T, \Delta)$ denotes the non-klt locus endowed with the closed subscheme structure defined by the multiplier ideal sheaf of $(T,\Delta)$. 
Then $D$ is semi-ample.
\end{thm}

\section{Proof of Theorem 1.1}
We will work over an algebraically closed field $k$. Suppose that $k \neq \overline{\mathbb{F}}_p$.
The goal of this section is to construct non-projective complete log canonical algebraic surfaces over $k$ with Kodaira dimension $0$, $1$ and $2$.
We start with the case $\kappa = 2$, which is the main highlight of this section. 

\subsection{\texorpdfstring{($\kappa=2$)}{kappa=2}}
The following construction is based on the example given in \cite[Section 4.1]{Sch99}.
Start with $\mathbb{P}^1 \times C$, where $C$ is an elliptic curve. Let $p_1\colon \mathbb{P}^1\times C\to \mathbb{P}^1$ and $p_2\colon \mathbb{P}^1\times C\to C$ be the first and second projections, respectively.
We take two distinct points $c_1, c_2\in C$.
Let $f \colon T \to \mathbb{P}^1 \times C$ be the blowup at $(0, c_1)$, $(\infty, c_2)$, and $(\infty, c_1)$. Let $E_{0, c_1}$, $E_{\infty, c_2}$, and $E_{\infty, c_1}$ be the corresponding exceptional divisors.
Let $B_1, B_2, C_0, C_1, C_\infty \subset T$ be the strict transforms of $H_1 \coloneqq \mathbb{P}^1 \times \{c_1\}$, $H_2 \coloneqq \mathbb{P}^1 \times \{c_2\}$, $F_0 \coloneqq \{0\} \times C$, $F_1 \coloneqq \{1\} \times C$, and $F_\infty \coloneqq \{\infty\} \times C$, respectively.

\begin{step}
We show that $B_1 + 3C_1$ is a big and semi-ample divisor on $T$, which induces a birational contraction morphism $g \colon T \to S$ onto a projective surface $S$ that contracts only $C_0$, $C_{\infty}$, and $E_{\infty,c_2}$.

To this end, we first rewrite the divisor as follows:
$$
\begin{aligned}
B_1 + 3C_1 &\sim B_1 + f^*F_0 + f^*F_\infty + C_1 \\
&\sim B_1 + C_0 + E_{0,c_1} + C_{\infty} + E_{\infty,c_1} + E_{\infty,c_2} + C_1 \\
&\sim f^*(H_1 + F_1) + C_0 + C_{\infty} + E_{\infty,c_2}.
\end{aligned}
$$
The divisor $H_1 + F_1$ on $\mathbb{P}^1 \times C$ is ample.
It follows that any irreducible curve on $T$ having non-positive intersection number with $B_1 + 3C_1$ must be either $f$-exceptional or contained in the support of $C_0 + C_{\infty} + E_{\infty,c_2}$.
To verify the nefness and identify the contracted curves, we compute the intersection numbers for all such candidates. Since
$$
(B_1 + 3C_1) \cdot C_0 = (B_1 + 3C_1) \cdot C_{\infty} = (B_1 + 3C_1) \cdot E_{\infty,c_2} = 0,
$$
and
$$
(B_1 + 3C_1) \cdot E_{0,c_1} = (B_1 + 3C_1) \cdot E_{\infty,c_1} = 1,
$$
we conclude that $B_1 + 3C_1$ is nef, and the irreducible curves whose intersection number with $B_1 + 3C_1$ equals $0$ are precisely $C_0$, $C_{\infty}$, and $E_{\infty,c_2}$.
Here, we recall that 
$$
K_{\mathbb{P}^1 \times C} \sim p_1^*K_{\mathbb{P}^1} + p_2^*K_C \sim -2F_1.$$
Then, the canonical divisor $K_T$ satisfies
\begin{align*}
K_T &\sim -2f^*F_1 + E_{0,c_1} + E_{\infty,c_1} + E_{\infty,c_2}\\ 
&\sim -2f^*F_1 + f^*F_0 + f^*F_{\infty} - C_0 - C_{\infty}\\
&\sim -C_0 - C_{\infty}.
\end{align*}
This implies that
$$
(B_1 + 3C_1) - (K_T + C_0 + C_{\infty}) \sim B_1 + 3C_1,
$$
which is nef and big. Since $T$ is smooth and $C_0+C_\infty$ is a reduced simple
normal crossing divisor, $\mathrm{Nklt}(T,C_0+C_\infty)
=C_0\sqcup C_\infty$ scheme-theoretically.
Since the support of $B_1+3C_1$ is disjoint from
$C_0\sqcup C_\infty$, we have
\[
\mathcal O_T(B_1+3C_1)|_{C_0\sqcup C_\infty}
\cong \mathcal O_{C_0\sqcup C_\infty}.
\]
Thus the restriction to the non-klt locus is semi-ample. Therefore, by Theorem~\ref{thm: basepoint-free}, $B_1 + 3C_1$ is semi-ample.
Consequently, the divisor $B_1 + 3C_1$ on $T$ is big and semi-ample, and the corresponding birational morphism $g \colon T \to S$ contracts only $C_0$, $C_{\infty}$, and $E_{\infty,c_2}$.
The resulting surface $S$ is projective.
\end{step}

\begin{step}
Our goal in this step is to construct, by gluing, a birational contraction $h \colon T \to Z$ contracting precisely $C_0$ and $C_\infty$, through which the morphism $g \colon T \to S$ factors. 

Let $U_0$ be the open subset of $T$ defined by
\[
U_0 \coloneqq (p_1 \circ f)^{-1}(\mathbb{P}^1 \setminus \{\infty\}) = T \setminus (C_\infty \cup E_{\infty, c_1} \cup E_{\infty, c_2}),
\]
and let $\varphi_0 \colon U_0 \to \mathbb{P}^1 \setminus \{\infty\}$ be the restriction of $p_1 \circ f$ to $U_0$. A direct calculation shows that the restriction
\[
(K_T + C_0 + B_1)|_{U_0} \sim (-2C_1 + E_{0, c_1} + C_0 + B_1)|_{U_0} \sim (f^*(-2F_1 + H_1) + C_0)|_{U_0}
\]
is $\varphi_0$-nef and $\varphi_0$-big. Thus, Theorem~\ref{thm: abundance} applied to the pair $(U_0, (C_0 + B_1)|_{U_0})$ implies that $(K_T + C_0 + B_1)|_{U_0}$ is $\varphi_0$-semi-ample. This induces a relative birational contraction morphism $h_0 \colon U_0 \to Z_0$ over $\mathbb{P}^1 \setminus \{\infty\}$. Since 
\[
(K_T + C_0 + B_1)|_{U_0} \cdot E_{0, c_1}|_{U_0} = 1 \quad \text{and} \quad (K_T + C_0 + B_1)|_{U_0} \cdot C_0|_{U_0} = 0,
\]
the morphism $h_0$ contracts precisely $C_0$.
Similarly, let $U_\infty$ be the open subset of $T$ defined by
\[
U_\infty \coloneqq (p_1 \circ f)^{-1}(\mathbb{P}^1 \setminus \{0\}) = T \setminus (C_0 \cup E_{0, c_1}),
\]
and let $\varphi_\infty \colon U_\infty \to \mathbb{P}^1 \setminus \{0\}$ be the restriction of $p_1 \circ f$ to $U_\infty$. The restriction
\begin{align*}
(K_T + C_\infty + B_1 + B_2)|_{U_\infty} &\sim (-2C_1 + E_{\infty, c_1}+E_{\infty, c_2} + C_\infty + B_1 + B_2)|_{U_\infty}\\
&\sim (f^*(-2F_1 + H_1 + H_2) + C_\infty)|_{U_\infty}
\end{align*}
is $\varphi_\infty$-nef and $\varphi_\infty$-big. By the relative abundance theorem for the pair $(U_\infty, (C_\infty + B_1 + B_2)|_{U_\infty})$, the divisor $(K_T + C_\infty + B_1 + B_2)|_{U_\infty}$ is $\varphi_\infty$-semi-ample, yielding a relative birational contraction morphism $h_\infty \colon U_\infty \to Z_\infty$ over $\mathbb{P}^1 \setminus \{0\}$. Since
\[
(K_T + C_\infty + B_1 + B_2)|_{U_\infty} \cdot E_{\infty, c_1}|_{U_\infty} = (K_T + C_\infty + B_1 + B_2)|_{U_\infty} \cdot E_{\infty, c_2}|_{U_\infty} = 1
\]
and $(K_T + C_\infty + B_1 +B_2)|_{U_\infty} \cdot C_\infty|_{U_\infty} = 0$, the morphism $h_\infty$ contracts precisely $C_\infty$.
By gluing $h_0$ and $h_\infty$ together, we obtain a contraction morphism $h \colon T \to Z$ over $\mathbb{P}^1$ which contracts only $C_0$ and $C_\infty$, and a normal complete surface $Z$ (cf.~\cite[II, Corollary 4.8]{Har77}).
Finally, since every curve contracted by $h\colon T\to Z$ is also contracted by $g\colon T\to S$, the morphism $g$ factors through $h$. This induces a birational morphism $\phi \colon Z \to S$ such that $g = \phi \circ h$.
\end{step}

\begin{step} \label{step: Z is non-proj}
We show that $Z$ is a complete log canonical surface with $K_Z\sim 0$, and that it is non-projective if we further arrange the choice of $c_1, c_2\in C$ so that $\mathcal O_C(c_1-c_2)\in\operatorname{Pic}^0(C)$ is a non-torsion element. The existence of such a choice is guaranteed by our assumption on $k$ and Fact~\ref{fact: non-torsion}.

The linear equivalence $K_T + C_0 + C_\infty \sim 0$ implies $K_Z \sim 0$, and $K_T = h^*(K_Z) - C_0 - C_\infty$ implies that $Z$ is log canonical.
To show that $Z$ is non-projective, we fix a choice of $c_1$ and $c_2$ such that $\mathcal O_C(c_1-c_2) \in \operatorname{Pic}^0(C)$ is a non-torsion element, and suppose for contradiction that there exists an ample divisor on $Z$. 
Then, there exists a very ample effective divisor $D$ on $Z$ whose support avoids the points $h(C_0)$ and $h(C_\infty)$ on $Z$.
Since $h$ does not contract $E_{0, c_1}$, $E_{\infty, c_1}$ or $E_{\infty, c_2}$, we have $h^*D\cdot E_{0, c_1}>0$, $h^*D\cdot E_{\infty, c_1}>0$, and $h^*D\cdot E_{\infty, c_2}>0$.
Let $\tilde{D}:=f_*h^*D$.
Here, we can write $f^*\tilde{D}=h^*D+n_1E_{0, c_1}+n_2E_{\infty, c_1}+n_3E_{\infty, c_2}$ for some integers $n_1, n_2, n_3$.
Since the $f$-exceptional divisors are disjoint from each other, we have the equation $0=f^*\tilde{D}\cdot E_{0, c_1}=h^*D\cdot E_{0, c_1}-n_1$, which implies $n_1>0$. Similarly, we obtain $n_2, n_3>0$.
By our choice of $D$, the effective divisor $h^*D$ does not contain $C_0$ or $C_\infty$ in its support, which implies that $\tilde{D}$ does not contain the fibers $F_0$ or $F_\infty$.
Hence the intersections satisfy $\tilde{D} \cap F_0 = \{(0, c_1)\}$ and $\tilde{D} \cap F_\infty = \{(\infty, c_1), (\infty, c_2)\}$ set-theoretically.
Since $F_0$ and $F_\infty$ are isomorphic via $f$ to the curves $C_0$ and $C_\infty$ respectively, identifying $F_0$ and $F_\infty$ with the curve $C$, the scheme-theoretic restrictions yield the relations $\tilde{D}|_{F_0} = n_1 c_1$ and $\tilde{D}|_{F_\infty} = n_2 c_1 + n_3 c_2$ as divisors on $C$.
We note that $\mathcal O_{\mathbb{P}^1\times C}(\tilde D)\cong p^*_1\mathcal L_1\otimes p^*_2\mathcal L_2$ for some invertible sheaves $\mathcal{L}_1$ on $\mathbb{P}^1$ and $\mathcal{L}_2$ on $C$.
This implies the linear equivalence $\tilde{D}|_{F_0} \sim \tilde{D}|_{F_\infty}$ as divisors on $C$. Therefore, we obtain $(n_1 - n_2)c_1 - n_3c_2 \sim 0$. 
By comparing the degrees of these two restrictions, we have $n_1 = n_2 + n_3$, which yields $n_3(c_1 - c_2) \sim 0$.
Since $\mathcal O_C(c_1 - c_2)$ is a non-torsion element, we must have $n_3 = 0$. This contradicts the fact that $n_3 > 0$. Consequently, $Z$ is non-projective.
\end{step}

From now on, we assume that the points $c_1$ and $c_2$ are chosen so that $Z$ is non-projective.

\begin{step} \label{step: cyclic cover}
We construct a finite cyclic cover $\tilde{Z} \to Z$ branched along a smooth divisor pulled back from $S$, and show that $\tilde{Z}$ is the desired surface by exploiting the non-projectivity of $Z$.

Let $J \subset S$ be the finite set consisting of the singular points of $S$. We choose a very ample Cartier divisor $A$ on the projective surface $S$.
For an integer $m\geq 2$, there exists a member $A' \in |mA|$ that also avoids $J$ and has a smooth support.
We define $A'' \coloneqq \phi^*A'$.
By construction, $A''$ is an effective Cartier divisor on $Z$ that is naturally identified with $A'$ via $\phi$. In particular, $A''$ is smooth and avoids the singular locus of $Z$. 
Next, we consider the cyclic cover of degree $m$ defined by
\[
\mu \colon \tilde{Z} \coloneqq \operatorname{Spec}_Z \left( \bigoplus_{l=0}^{m-1} \mathcal{O}_Z(-l\phi^*A) \right) \to Z,
\]
which is ramified along $A''$ (cf.~\cite[Definition 2.50]{KM98}). Here, we set $m = 2$ if the base field $k$ has characteristic zero; in positive characteristic, we choose $m$ to be coprime to $\operatorname{char}(k)$.
By the ramification formula for finite covers, we obtain the equality of $\mathbb{Q}$-divisors:
\[
K_{\tilde{Z}} = \mu^* \left( K_Z + \frac{m-1}{m}A'' \right).
\]
Since $Z$ is log canonical and $A''$ is a smooth divisor avoiding the singular points of $Z$, the pair $(Z, \frac{m-1}{m}A'')$ is log canonical. Consequently, the finiteness of $\mu$ together with \cite[Proposition 5.20]{KM98} implies that $\tilde{Z}$ is a log canonical surface.
If there were an ample invertible sheaf $\mathcal{L}$ on $\tilde{Z}$, its image $\operatorname{Norm}_\mu(\mathcal{L})$ via the norm map associated with the cyclic cover $\mu \colon \tilde{Z} \to Z$ would yield an ample invertible sheaf on $Z$ (cf.~\cite[Tag 0BD4]{stacks-project}).
This contradicts the fact that $Z$ is non-projective.
Hence, $\tilde{Z}$ is a non-projective surface.
Since $K_Z \sim 0$, the ramification formula yields $\mathcal{O}_{\tilde{Z}}(K_{\tilde{Z}}) \cong \mu^*\phi^* \mathcal{O}_S((m-1)A)$.
Since the divisor $A$ is very ample on the projective surface $S$, the canonical divisor $K_{\tilde{Z}}$ is semi-ample.
Furthermore, the Kodaira dimension satisfies
$$
\kappa(\tilde{Z}, K_{\tilde{Z}}) = \kappa(S, \mathcal{O}_{S}((m-1)A)) = 2.
$$
Therefore, $\tilde{Z}$ is a normal non-projective complete log canonical surface with a semi-ample canonical divisor $K_{\tilde{Z}}$ and Kodaira dimension $\kappa(\tilde{Z}, K_{\tilde{Z}}) = 2$.
\end{step}

\subsection{\texorpdfstring{($\kappa=0, 1$)}{kappa=0, 1}}
Examples with $\kappa=0$ can be essentially found in \cite[Section 12]{fujino-nagoya}, \cite[Aside 3.46]{Kol07}, and \cite[Section 2.5]{Sch99}. 
Here, for convenience, we construct examples with $\kappa=0$ and $1$ by slightly modifying the case $\kappa=2$ above.
First, Step~\ref{step: Z is non-proj} directly implies that the surface $Z$ itself is the desired example with $\kappa(Z, K_Z)=0$.
Next, we construct an example with $\kappa=1$. In the construction of Step~\ref{step: cyclic cover}, we use a general fiber of the morphism $Z \to \mathbb{P}^1$ instead of the smooth divisor pulled back from the ample divisor $A$ on $S$. Then, we consider a cyclic cover ramified along $m$ general fibers. For the resulting surface $\tilde{Z}$, the invertible sheaf $\mathcal{O}_{\tilde{Z}}(K_{\tilde{Z}})$ is isomorphic to the pull-back of $\mathcal{O}_{\mathbb{P}^1}(m-1)$. This implies that $K_{\tilde{Z}}$ is semi-ample and satisfies $\kappa(\tilde{Z}, K_{\tilde{Z}})=1$.

\section{Proof of Theorem 1.2}

The first author proved a theorem similar to Theorem~\ref{thm1.2} in \cite[Theorem~1.3]{fujino-nagoya}. 
Here, for the reader's convenience, we give a relatively elementary proof of Theorem \ref{thm1.2} within the framework of the minimal model program. We first show that $X$ has at most one non-rational singular point (Corollary \ref{cor2.2}). Then, by a classical argument due to Goodman (see 
\cite{goodman}), we obtain the projectivity of $X$. Note that the argument used in the proof of \cite[Theorem~1.3]{fujino-nagoya} does not directly apply to our setting, since a sufficiently developed minimal model theory for two-dimensional algebraic spaces is not currently available.

\begin{proof}[Proof of Theorem~\ref{thm1.2}]
Let $f\colon Y\to X$ be the minimal resolution of singularities.
Then $Y$ is a smooth projective surface with
$\kappa(Y,K_Y)=-\infty$. 
Since $X$ is a normal surface, it is Cohen--Macaulay.
Moreover, the assumption $\kappa(X,K_X)=-\infty$ implies that
\[
H^0(X,\mathcal O_X(K_X))=0.
\]
Hence Serre duality (see, for example, \cite[Tag 0FW0]{stacks-project}) yields
\[
H^2(X,\mathcal O_X)=0.
\]
Therefore, the Leray spectral sequence gives the following short exact sequence:
\begin{equation}\label{eq1}
0
\to H^1(X,\mathcal O_X)
\to H^1(Y,\mathcal O_Y)
\to H^0(X,R^1f_*\mathcal O_Y)
\to 0.
\end{equation}
\setcounter{step}{0}

\begin{step}\label{step1}
We show that $X$ has at most one non-rational singular point.

Assume that $X$ has a non-rational singular point $P\in X$.
If $Y$ is rational, namely birational to $\mathbb P^2$,
then
\[
H^1(Y,\mathcal O_Y)=0.
\]
By \eqref{eq1}, this implies that
\[
R^1f_*\mathcal O_Y=0,
\]
and hence all singularities of $X$ are rational, a contradiction.
Therefore, $Y$ is irrational. Running the minimal model program for $Y$, we may assume that
$Y$ is obtained from a $\mathbb P^1$-bundle
\(
\mathbb P_C(\mathcal E)\to C
\)
over a smooth projective curve $C$ with $g(C)\ge 1$
by a sequence of blow-ups. 
By the classification of two-dimensional log canonical singularities,
the exceptional divisor $f^{-1}(P)$ cannot be mapped to a point of $C$.
Indeed, every non-rational two-dimensional log canonical singularity is either a simple elliptic singularity or a cusp singularity.
Hence $f^{-1}(P)$ dominates $C$.
It follows that $C$ is either an elliptic curve or $\mathbb P^1$. Since $P$ is non-rational, we have
\[
R^1f_*\mathcal O_Y\ne 0.
\]  
Thus \eqref{eq1} implies that
\[
\dim H^1(Y,\mathcal O_Y)\ne 0.
\]
Therefore $g(C)=1$, and consequently
\[
\dim H^1(Y,\mathcal O_Y)=1.
\]
Again by \eqref{eq1}, the sheaf $R^1f_*\mathcal O_Y$
is supported only at $P$.
Hence $X$ has at most one non-rational singular point.
\end{step}

\begin{step}[Goodman]
We prove that $X$ is projective.

By Step~\ref{step1}, there exists an affine open subset
$U\subset X$
such that every point of $X\setminus U$ is a rational singularity. 
In particular, every irreducible component of $X\setminus U$ is a $\mathbb Q$-Cartier Weil divisor. 
Then, by the same argument as in the proof of
\cite[Theorem~2]{goodman} (see also \cite[Theorem 1.28 (Zariski--Goodman)]{bad-book}), 
one can construct an effective Cartier divisor $D$ on $X$
such that
\[
\operatorname{Supp} D = X\setminus U
\]
and $D$ is ample.
Therefore $X$ is projective.
\end{step}

This completes the proof.
\end{proof}

Step~\ref{step1} in the proof of Theorem~\ref{thm1.2}
establishes the following corollary.

\begin{cor}\label{cor2.2}
Let $X$ be a complete log canonical algebraic surface over an algebraically closed field.
If $\kappa(X,K_X)=-\infty$, then $X$ has at most one non-rational singular point.
\end{cor}

\begin{rem}\label{criterion}
Theorem~\ref{thm1.2} follows easily from Schr\"oer's projectivity criterion \cite[Corollary 4.4]{Sch00}: a complete normal surface $X$ with $H^2(X, \mathcal{O}_X)=0$ is projective. As we saw in the first part of the proof of Theorem~\ref{thm1.2}, this vanishing follows from Serre duality. Over $\mathbb{C}$, this projectivity criterion also follows from \cite[Lemma 9]{Bre81}. As mentioned earlier, in this paper we have provided a proof of Theorem~\ref{thm1.2} by combining well-known facts about log canonical surface singularities with a classical argument. As a byproduct of this proof, we see that $X$ has at most one non-rational singular point, as stated in Corollary~\ref{cor2.2}.
\end{rem}

We conclude this paper with a remark on Nagata's example.

\begin{rem}
In \cite[Example 1]{nagata}, Nagata constructed a non-projective complete surface $S^*$. By construction, $S^*$ is birational to $\mathbb{P}^2$ and is not log canonical. Furthermore,
$H^0(S^*,\mathcal O_{S^*}(K_{S^*}))\neq 0$, 
which can be verified by an application of the Leray spectral sequence as in the proof of Theorem~\ref{thm1.2}.
\end{rem}

\bibliographystyle{amsalpha}
\bibliography{ref}

\providecommand{\bysame}{\leavevmode\hbox to3em{\hrulefill}\thinspace}
\providecommand{\MR}{\relax\ifhmode\unskip\space\fi MR }
\providecommand{\MRhref}[2]{%
  \href{http://www.ams.org/mathscinet-getitem?mr=#1}{#2}
}
\providecommand{\href}[2]{#2}
\begin{thebibliography}{{Sta}26}

\bibitem[B{\u{a}}d01]{bad-book}
Lucian B{\u{a}}descu, \emph{Algebraic surfaces}, Universitext, Springer-Verlag, New York, 2001, Translated from the 1981 Romanian original by Vladimir Ma\c{s}ek and revised by the author. \MR{1805816}

\bibitem[Bre81]{Bre81}
Lawrence Brenton, \emph{On singular complex surfaces with vanishing geometric genus, and pararational singularities}, Compositio Math. \textbf{43} (1981), no.~3, 297--315. \MR{632431}

\bibitem[Fuj12]{fujino-surface}
Osamu Fujino, \emph{Minimal model theory for log surfaces}, Publ. Res. Inst. Math. Sci. \textbf{48} (2012), no.~2, 339--371. \MR{2928144}

\bibitem[Fuj21a]{fujino-nagoya}
\bysame, \emph{Minimal model theory for log surfaces in {F}ujiki's class {${\mathcal C}$}}, Nagoya Math. J. \textbf{244} (2021), 256--282. \MR{4335910}

\bibitem[Fuj21b]{Fuj21}
\bysame, \emph{On minimal model theory for algebraic log surfaces}, Taiwanese J. Math. \textbf{25} (2021), no.~3, 477--489. \MR{4298910}

\bibitem[Goo69]{goodman}
Jacob~Eli Goodman, \emph{Affine open subsets of algebraic varieties and ample divisors}, Ann. of Math. (2) \textbf{89} (1969), 160--183. \MR{242843}

\bibitem[Har77]{Har77}
Robin Hartshorne, \emph{Algebraic geometry}, Graduate Texts in Mathematics, vol.~52, Springer-Verlag, New York-Heidelberg, 1977. \MR{463157}

\bibitem[KM98]{KM98}
J\'anos Koll\'ar and Shigefumi Mori, \emph{Birational geometry of algebraic varieties}, Cambridge Tracts in Mathematics, vol. 134, Cambridge University Press, Cambridge, 1998, With the collaboration of C. H. Clemens and A. Corti, Translated from the 1998 Japanese original. \MR{1658959}

\bibitem[Kol07]{Kol07}
J\'anos Koll\'ar, \emph{Lectures on resolution of singularities}, Annals of Mathematics Studies, vol. 166, Princeton University Press, Princeton, NJ, 2007. \MR{2289519}

\bibitem[Mor26]{Mor26}
Nao Moriyama, \emph{A note on {$\mathbb{Q}$}-{G}orenstein surfaces}, 2026, \href{https://arxiv.org/abs/2602.00732}{arXiv:2602.00732}.

\bibitem[Nag58]{nagata}
Masayoshi Nagata, \emph{Existence theorems for nonprojective complete algebraic varieties}, Illinois J. Math. \textbf{2} (1958), 490--498. \MR{97406}

\bibitem[Sch99]{Sch99}
Stefan Schr\"oer, \emph{On non-projective normal surfaces}, Manuscripta Math. \textbf{100} (1999), no.~3, 317--321. \MR{1726231}

\bibitem[Sch00]{Sch00}
\bysame, \emph{On contractible curves on normal surfaces}, J. Reine Angew. Math. \textbf{524} (2000), 1--15.

\bibitem[{Sta}26]{stacks-project}
The {Stacks Project Authors}, \emph{The {Stacks Project}}, \url{https://stacks.math.columbia.edu}, 2026.

\bibitem[Tan14]{tanaka}
Hiromu Tanaka, \emph{Minimal models and abundance for positive characteristic log surfaces}, Nagoya Math. J. \textbf{216} (2014), 1--70. \MR{3319838}

\end{thebibliography}

\end{document}